\documentclass{article}
\usepackage[cp1251]{inputenc}
\usepackage[russian]{babel}
\usepackage{amssymb,amsfonts,amsmath,amsthm,amscd}
\usepackage{graphicx}
\usepackage{enumerate}
\usepackage{soul}
\usepackage[matrix,arrow,curve]{xy}
\usepackage{longtable}
\usepackage{array}
\usepackage{bm}

\newdimen\symskip
\newdimen\defskip
\newdimen\parind
\parind=\parindent
\newdimen\leftmarge
\newdimen\theoremshape
\topsep\defskip
\makeatletter

\newcommand*{\clei}{\nobreak\hskip\z@skip}

\renewcommand{\:}{\textup{:}}
\renewcommand{\~}{\textup{;}}
\DeclareRobustCommand*{\ti}{~\textemdash{} }
\DeclareRobustCommand*{\dh}{\clei\hbox{-}\clei}

\def\thempfn{\ifcase\value{footnote}1\or *\or **\or ***\else\@ctrerr\fi}
\renewcommand\footnoterule{%
  \kern-3\p@
  \hrule\@width1in
  \kern2.6\p@}
\@addtoreset{footnote}{section}
\@addtoreset{footnote}{page}

\renewcommand{\@biblabel}[1]{[#1]}
\renewenvironment{thebibliography}[1]
     {\renewcommand{\refname}{References}%
      \section*{\refname}%
      \@mkboth{\MakeUppercase\refname}{\MakeUppercase\refname}%
      \list{\@biblabel{\@arabic\c@enumiv}}%
           {\itemsep\baselineskip
            \leftmargin\parind
            \settowidth\labelwidth{\@biblabel{#1}}%
            \labelsep\parind\advance\labelsep-\labelwidth
            \@openbib@code
            \usecounter{enumiv}%
            \let\p@enumiv\@empty
            \renewcommand\theenumiv{\@arabic\c@enumiv}}%
      \sloppy
      \clubpenalty4000
      \@clubpenalty\clubpenalty
      \widowpenalty4000%
      \sfcode`\.\@m}
     {\def\@noitemerr
       {\@latex@warning{Empty `thebibliography' environment}}%
      \endlist}

\def\@maketitle{%
  \newpage
  \vskip0.5em%
  MSC \msc%
  \vskip1.5em%
  \begin{center}\bf%
  \let\footnote\thanks%
   {\Large\@author\par}%
   \vskip1em%
   {\LARGE\@title\par}%
   \vskip1em%
   {\large\@date}%
  \end{center}%
  \par
  \vskip1.5em}

\def\@title{\@latex@warning@no@line{No \noexpand\title given}}

\renewcommand\sectionmark[1]{%
 \markright{%
  \ifnum \c@secnumdepth >\z@
   \thesection. \ %
  \fi
 #1}}%
\renewcommand{\section}{\@startsection{section}{1}{0pt}%
{5.5ex plus .5ex minus .2ex}{1.5ex plus .3ex}%
{\center\normalfont\Large\bfseries\sffamily\boldmath}}
\newcommand{\Ss}{\textup{\S\,}}

\def\@postskip@{\hskip.5em\relax}

\def\postsection{.\@postskip@}
\def\@seccntformat#1{\csname pre#1\endcsname\csname the#1\endcsname\csname post#1\endcsname}
\renewcommand{\thesection}{\textup{\arabic{section}}}

\newcommand{\parr}{\par\addvspace{\defskip}}
\newcommand{\theo}[2]{\newtheorem{#1}{#2}}
\newcommand{\deff}[2]{\newenvironment{#1}{\parr\textbf{#2.}}{\parr}}
\theo{theorem}{Theorem}
\theo{lemma}{Lemma}
\theo{prop}{Proposition}
\theo{stm}{Statement}
\theo{imp}{Corollary}
\deff{df}{Definition}
\deff{note}{Note}

\def\@begintheorem#1#2[#3]{%
  \deferred@thm@head{\the\thm@headfont \thm@indent
    \@ifempty{#1}{\let\thmname\@gobble}{\let\thmname\@iden}%
    \@ifempty{#2}{\let\thmnumber\@gobble}{\let\thmnumber\@iden}%
    \@ifempty{#3}{\let\thmnote\@gobble}{\let\thmnote\@iden}%
    \thm@notefont{\bfseries\upshape}%
    \indent%
    \thm@swap\swappedhead\thmhead{#1}{#2}{#3}%
    \the\thm@headpunct
    \thmheadnl 
    \hskip\thm@headsep
  }%
  \ignorespaces}
\renewenvironment{proof}{\parr\pushQED{\qed}\normalfont$\square\quad$}{\popQED\@endpefalse\parr}

\newcommand{\labheadi}[1]{\textup{#1)}}
\renewcommand{\labelenumi}{\labheadi{\arabic{enumi}}}
\renewcommand{\theenumi}{\labheadi{\arabic{enumi}}}

\newcounter{col}
\newcounter{coll}
\newcommand{\mt}[3]{\multicolumn{#1}{#2}{#3}}
\newcommand{\news}{\\\hline}
\newcommand{\refcol}[1]{\addtocounter{col}{#1}}
\newcommand{\fcol}{\setcounter{coll}{\value{col}}}
\newcommand{\mc}[3]{&\mco{#1}{#2}{#3}}
\newcommand{\nc}[2]{\mc{1}{#1}{#2}}
\newcommand{\mco}[3]{\mt{#1}{#2|@{\refcol{#1}}}{#3}}
\newcommand{\mcu}[3]{\mco{#1}{|#2}{#3}}
\newcommand{\ncu}[2]{\mcu{1}{#1}{#2}}
\newcommand{\mtco}[2]{\mt{\value{coll}}{#1}{#2}}
\newcommand{\mtc}[1]{\mtco{c}{#1}}

\newcommand{\nwl}{\newline}
\newcommand{\lonu}[3]{%
\setcounter{col}{0}
\begin{longtable}{#1}
\hline#2
\fcol
\endfirsthead
\hline#2
\endhead
\mtc{}\\
\mtc{\textit{Continuation on the next page}}
\endfoot
\endlastfoot
\hline#3\news\end{longtable}}

\makeatother

\newcommand{\sm}{\setminus}
\newcommand{\cln}{\colon}
\newcommand{\nl}{\lhd}
\newcommand{\ol}{\overline}
\renewcommand{\ge}{\geqslant}
\renewcommand{\le}{\leqslant}

\newcommand{\subs}{\subset}
\newcommand{\sups}{\supset}
\newcommand{\Ra}{\Rightarrow}
\newcommand{\thra}{\twoheadrightarrow}

\newcommand{\cups}[1]{\bigcup\limits_{{#1}}}

\newcommand{\sco}{,\ldots,}

\newcommand{\br}[1]{\bigl(#1\bigr)}

\newcommand{\ter}[1]{\textup{(}#1\textup{)}}
\newcommand{\bgm}[1]{\bigl|#1\bigr|}
\newcommand{\Bm}[1]{\Bigl|#1\Bigr|}

\newcommand{\bs}[1]{\bigl[#1\bigr]}
\newcommand{\bc}[1]{\bigl\{#1\bigr\}}
\newcommand{\case}[1]{\begin{cases}#1\end{cases}}
\newcommand{\Ga}{\Gamma}

\newcommand{\mbb}{\mathbb}
\newcommand{\Z}{\mbb{Z}}

\begin{document}

\author{Styrt O.\,G.\thanks{Russia, MIPT, oleg\_styrt@mail.ru}}
\title{Groups~$\Ga_n^4$\: algebraic properties}
\date{}
\newcommand{\udk}{?}
\newcommand{\msc}{20F36}

\maketitle

In the paper, groups~$\Ga_n^4$ closely connected with braid groups are researched from algebraic point of view.
More exactly, for $n\ge7$, it is proved that $\Ga_n^4$ is a~nilpotent finite $2$\dh group with $4$\dh torsion
and that its subgroup~$(\Ga_n^4)'$ is central.

\smallskip


\section*{The notations used}

\lonu{|>{$}l<{$}|l|}
{\ncu{>{\nwl}p{2cm}<{\nwl}}{Expression} \nc{>{\nwl}p{9cm}<{\nwl}}{Meaning}}{%
[n] & $\{1\sco n\}$
\news
C_M^k & $\bc{M'\subs M\cln|M'|=k}$
\news
[a,b] & $aba^{-1}b^{-1}$
\news
Z(a) & $\bc{b\cln[a,b]=e}$}

\section{Introduction}\label{introd}

The paper~\cite{Manick} is devoted to researching groups~$\Ga_n^4$ that are closely connected with braid groups.
The group~$\Ga_n^4$ is defined in~\cite[\Ss2]{Manick} by
\begin{itemize}
\item the generators~$d_{(ijkl)}$ for all pairwise distinct $i,j,k,l\in[n]$\~
\item the relations 1--4.
\end{itemize}
For convenience, denote by~$S$ the generator set and by~$s_{ik}^{jl}$ the generator~$d_{(ijkl)}$. For $P\in C_{[n]}^4$,
set $S_P:=\bc{s_{ik}^{jl}\cln\{i,j,k,l\}=P}\subs S$. Then, the relations \ref{inv}--\ref{sym} have the following sense\:
\begin{enumerate}
\renewcommand{\labheadi}[1]{\textup{#1}}
\renewcommand{\labelenumi}{\labheadi{\arabic{enumi}}\textup{.}}
\renewcommand{\theenumi}{\labheadi{\arabic{enumi}}}
\item\label{inv} all elements of~$S$ are involutions\~
\item\label{comm} if $P,Q\in C_{[n]}^4$ and $|P\cap Q|\le2$, then $[S_P,S_Q]=\{e\}$\~
\addtocounter{enumi}{1}
\item\label{sym} $s_{ik}^{jl}=s_{ki}^{jl}=s_{ik}^{lj}=s_{jl}^{ik}$\ti so, each element~$s_{ik}^{jl}$ is uniquely defined by the \textit{unordered} pair $\bc{\{i,k\},\{j,l\}}$ of the disjoint \textit{unordered} pairs $\{i,k\}$ and $\{j,l\}$\~
\addtocounter{enumi}{-1}
\addtocounter{enumi}{-1}
\item\label{pent} if $\{i,j,k,l,m\}\in C_{[n]}^5$, then $s_{ik}^{jl}s_{il}^{jm}s_{jl}^{km}s_{ik}^{jm}s_{il}^{km}=e$, i.\,e.
$s_{ki}^{lj}s_{il}^{jm}s_{lj}^{mk}s_{jm}^{ki}s_{mk}^{il}=e$.
\end{enumerate}
Hence, the condition~\ref{pent} means that any pairwise distinct numbers $k_1\sco k_5\in[n]$ satisfy
\begin{equation}\label{pen}
s_{k_1k_2}^{k_3k_4}\cdot s_{k_2k_3}^{k_4k_5}\cdot s_{k_3k_4}^{k_5k_1}\cdot s_{k_4k_5}^{k_1k_2}\cdot s_{k_5k_1}^{k_2k_3}=e.
\end{equation}

The main result of the paper is the following theorem.

\begin{theorem}\label{main} If $n\ge7$, then
\begin{itemize}
\item the subgroup~$(\Ga_n^4)'$ of~$\Ga_n^4$ is central\~
\item $(\Ga_n^4)'\cong(\Z_2)^p$ and $\Ga_n^4/(\Ga_n^4)'\cong(\Z_2)^q$, where $p\le C_n^3$ and $q\le3\cdot C_n^4$.
\end{itemize}
\end{theorem}

\begin{imp}\label{submain} If $n\ge7$, then
\begin{itemize}
\item $\Ga_n^4$ is a~nilpotent finite $2$\dh group\~
\item $\bs{\Ga_n^4,(\Ga_n^4)'}=\{e\}$\~
\item $a^4=e$ for each $a\in\Ga_n^4$.
\end{itemize}
\end{imp}

\section{Proofs of the results}\label{prove}

In this section, Theorem~\ref{main} is proved.

For $a,x,y\in\Ga_n^4$, write $x\sim_a y$ if the following (obviously, equivalent) conditions hold\:
\begin{enumerate}
\item $x^{-1}y\in Z(a)$\~
\item $xax^{-1}=yay^{-1}$\~
\item $[x,a]=[y,a]$\~
\item $[a,x]=[a,y]$.
\end{enumerate}
It is clear that, with $a$ fixed, the \textit{binary} relation~$\sim_a$ on \textit{two} elements $x,y$ is an equivalence.

From now, we will assume that $n\ge6$.

\begin{lemma}\label{eqv} If $P\in C_{[n]}^4$, $K\in C_P^3$, and $s\in S_P$,
then $s_1\sim_s s_2$ for all $s_1,s_2\in\cups{k'\notin P}S_{K\sqcup\{k'\}}$.
\end{lemma}

\begin{proof} For $k'\notin K$ and $k\in K$, set $t_{k'k}:=s_{k'k}^{ij}$ where $\{i,j\}=K\sm\{k\}$.

Note that $\bgm{[n]\sm P}=n-4\ge2$. Thus, it suffices to prove for any
distinct $k_1,k_2\in[n]\sm P$ and distinct $i,j\in K$ the relation $t_{k_1i}\sim_s t_{k_2j}$.

In~$K$, take arbitrary distinct elements $k_3,k_5$ and denote the rest one by~$k_4$.
Each subset~$Q$ of type $\br{K\sm\{k\}}\sqcup\{k_1,k_2\}$ ($k\in K$) satisfies
$Q\in C_{[n]}^4$ and $P\cap Q\subs Q\sm\{k_1,k_2\}=K\sm\{k\}$ that implies $S_Q\subs Z(s)$. By~\eqref{pen},
$t_{k_2k_3}^{-1}t_{k_1k_5}=
s_{k_2k_3}^{k_4k_5}\cdot s_{k_3k_4}^{k_5k_1}=s_{k_1k_2}^{k_3k_4}\cdot s_{k_5k_1}^{k_2k_3}\cdot s_{k_4k_5}^{k_1k_2}\in Z(s)$.
\end{proof}

\begin{lemma}\label{com} Take any subset $K\in C_{[n]}^3$. Set $S_{(i)}:=S_{K\sqcup\{i\}}$ \ter{$i\notin K$}.
Then all elements of type $[s_i,s_j]$ \ter{$s_i\in S_{(i)}$, $s_j\in S_{(j)}$, $i,j\notin K$, $i\ne j$}
are the same involution~$c^{(K)}$ that is central for $n\ge7$.
\end{lemma}

\begin{proof} Set $M:=[n]\sm K$. Clearly, $|M|=n-3\ge3$.
If $k_0\in M$ and $s\in S_{(k_0)}$, then applying Lemma~\ref{eqv} to $P:=K\sqcup\{k_0\}$
gives for any $s_1,s_2\in\cups{k'\in M\sm\{k_0\}}S_{(k')}$ the relations $s_1s_2\in Z(s)$,
$[s_1,s]=[s_2,s]$, and $[s,s_1]=[s,s_2]$.
Therefore\:
\begin{itemize}
\item If $\{i,j\}\in C_M^2$, then all elements $[s_i,s_j]$ ($s_i\in S_{(i)}$, $s_j\in S_{(j)}$)
are the same element~$c^{(K)}_{ij}$. It is obvious that $c^{(K)}_{ji}=(c^{(K)}_{ij})^{-1}$.
\item If $k,i,j\in M$ and $i,j\ne k$, then $c^{(K)}_{ik}=c^{(K)}_{jk}$ and $c^{(K)}_{ki}=c^{(K)}_{kj}$.
\end{itemize}
Since $|M|\ge3$, all elements $c^{(K)}_{ij}$ ($i,j\in M$, $i\ne j$) are the same element $c^{(K)}=(c^{(K)})^{-1}$.
It remains to prove that $[c^{(K)},\Ga_n^4]=\{e\}$ for $n\ge7$.

Suppose that $n\ge7$, $P\in C_{[n]}^4$, and $s\in S_P$\~ show that $c^{(K)}\in Z(s)$.

There exist distinct numbers $k_1,k_2\in M$ such that $\br{|M\sm P|\ge2}\Ra(k_1,k_2\notin P)$.
Take any elements $s_i\in S_{(k_i)}$ ($i\in[2]$). Then $c^{(K)}=[s_1,s_2]=(s_1s_2)^2$.

Assume that $c^{(K)}\notin Z(s)$.

We have $s_1s_2\notin Z(s)$ and $s_i\notin Z(s)$ for some $i\in[2]$. Hence,
$\Bm{\br{K\sqcup\{k_i\}}\cap P}\ge3$. Thus,
\begin{itemize}
\item $|K\cap P|\ge2$\~
\item if $k_i\notin P$, then $|K\cap P|\ge3$, i.\,e. $P\sups K$.
\end{itemize}
Note that $|M\sm P|=n-|K\cup P|\ge7-|K\cup P|=|K\cap P|\ge2$. Therefore, $k_1,k_2\notin P$, $P\sups K$.
So, $P=K\sqcup\{k_0\}$ ($k_0\in M$, $k_0\ne k_1,k_2$). Hence, $s_1s_2\in Z(s)$, a~contradiction.

Thus, if $n\ge7$, then $[c^{(K)},S]=\{e\}$ and, therefore, $[c^{(K)},\Ga_n^4]=\{e\}$.
\end{proof}

From now, assume that $n\ge7$.

Due to Lemma~\ref{com}, the subgroup $H\subs\Ga_n^4$ generated by the involutions
$c^{(K)}$ ($K\in C_{[n]}^3$) is central; hence, $H\nl\Ga_n^4$ and $H\cong(\Z_2)^p$, $p\le C_n^3$.
Denote by~$\pi$ the factoring homomorphism $\Ga_n^4\thra\Ga_n^4/H$\~ write $\ol{x}$ for $\pi(x)$.

\begin{lemma}\label{codi} If $P,Q\in C_{[n]}^4$ and $P\ne Q$, then $[\ol{S}_P,\ol{S}_Q]=\{e\}\subs\Ga_n^4/H$.
\end{lemma}

\begin{proof} We have $[S_P,S_Q]\subs H$ since
\begin{equation*}
[S_P,S_Q]=\case{
\{e\},&|P\cap Q|\le2;\\
\{c^{(P\cap Q)}\},&|P\cap Q|=3.}\qedhere
\end{equation*}
\end{proof}

\begin{lemma}\label{coal} All elements $\ol{s}\in\Ga_n^4/H$ \ter{$s\in S$} pairwise commute.
\end{lemma}

\begin{proof} By Lemma~\ref{codi}, it suffices to prove that $[\ol{s}_{k_1k_2}^{k_3k_4},\ol{s}_{k_1k_4}^{k_3k_2}]=e$
for any pairwise distinct~$k_i$. Consider the set $P\in C_{[n]}^4$ of these~$k_i$ and an arbitrary number $k_5\notin P$.
For distinct $i',i\in[4]$, set $r_{i'i}:=\ol{s}_{k_5k_i}^{ll'}$ where $\{l,l'\}=P\sm\{k_i,k_{i'}\}$
(so, $\{k_5,k_i,l,l'\}=\br{P\sqcup\{k_5\}}\sm\{k_{i'}\}$). It follows from Lemma~\ref{codi} that
$[r_{i'i},r_{j'j}]=e$ whenever $i'\ne j'$. By~\eqref{pen}, $\ol{s}_{k_1k_2}^{k_3k_4}=r_{14}r_{21}r_{34}r_{41}$.
Replace $k_2$ and~$k_4$\: $\ol{s}_{k_1k_4}^{k_3k_2}=r_{12}r_{41}r_{32}r_{21}$.
Hence, $\ol{s}_{k_1k_2}^{k_3k_4}\cdot\ol{s}_{k_1k_4}^{k_3k_2}=r_{14}r_{12}r_{34}r_{32}$.
Now replace $k_1$ and~$k_3$\: $\ol{s}_{k_3k_2}^{k_1k_4}\cdot\ol{s}_{k_3k_4}^{k_1k_2}=r_{34}r_{32}r_{14}r_{12}$,
i.\,e.
$\ol{s}_{k_1k_4}^{k_3k_2}\cdot\ol{s}_{k_1k_2}^{k_3k_4}=r_{34}r_{32}r_{14}r_{12}=r_{14}r_{12}r_{34}r_{32}=
\ol{s}_{k_1k_2}^{k_3k_4}\cdot\ol{s}_{k_1k_4}^{k_3k_2}$.
\end{proof}

By Lemma~\ref{coal}, the group $\Ga_n^4/H$ is generated by pairwise commuting involutions $\ol{s}$ ($s\in S$);
hence, $\Ga_n^4/H\cong(\Z_2)^q$, $q\le|S|\le3\cdot C_n^4$. In particular, $(\Ga_n^4/H)'=\{e\}$, i.\,e. $(\Ga_n^4)'\subs H$.
Since $c^{(K)}\in(\Ga_n^4)'$ for each $K\in C_{[n]}^3$, then $H=(\Ga_n^4)'$.

Thus, we completely proved Theorem~\ref{main} and, hence, Corollary~\ref{submain}.

\section*{Acknowledgements}

The author is grateful to Prof. \fbox{E.\,B.\?Vinberg} for exciting interest to algebra.
The author is grateful to Prof. V.\,O.\?Manturov for introducing and attracting to his research area
and for useful discussions.

The author dedicates the article to E.\,N.\?Troshina.

\end{document}